\documentclass{article}
\usepackage[a4paper, total={5.3in, 8in}]{geometry}
\usepackage{graphicx} % Required for inserting images
\usepackage{amssymb, amsmath, amsthm}
\usepackage{xcolor}
\usepackage{mathptmx, helvet, courier, makeidx, graphicx, multicol, footmisc}
\usepackage{amsmath}
\usepackage{amssymb}
\usepackage{mathtools}

\newtheorem{theorem}{Theorem}[section]

\newtheorem{corollary}[theorem]{Corollary}

\newtheorem{lemma}[theorem]{Lemma}

\newtheorem{proposition}[theorem]{Proposition}
\newtheorem{remark}[theorem]{Remark}

\usepackage{color}

\title{Bounds on distinct and repeated dot product trees}
\author{Aaron Autry, Slade Gunter, Christopher Housholder, and Steven Senger}
\date{\today}

\begin{document}

\maketitle
\abstract{We study questions inspired by Erd\H os' celebrated distance problems with dot products in lieu of distances, and for more than a single pair of points. In particular, we study point configurations present in large finite point sets in the plane that are described by weighted trees. We give new lower bounds on the number of distinct sets of dot products serving as weights for a given type of tree in any large finite point set. We also as demonstrate the existence of many repetitions of some special sets of dot products occurring in a given type of tree in different constructions, narrowing gap between the best known upper and lower bounds on these configurations.}
\section{Introduction}
\subsection{Background}
In 1946, in \cite{Erd46}, Paul Erd\H os posed a problem that continues to inspire mathematicians: In a large finite set of points in the plane, what is the upper bound of the number pairs can determine the same distance? This is often called the unit distance problem. A related problem also described in \cite{Erd46} is the so-called distinct distances problem, which asks for a lower bound on the number of distinct distances that must be determined by any large finite set of points in the plane.

For the unit distance problem, after some initial activity, general progress seems to have halted in 1984, at the upper bound of a constant times $n^\frac{4}{3},$ given by Szemer\' edi and Trotter in \cite{SST}. In the case of the distinct distances problem, there was much activity and incremental progress up until 2010, when Guth and Katz essentially resolved the problem in \cite{GK}. For a more detailed look, see \cite{GIS} and \cite{GK}, and the references contained therein.

Though Erd\H os originally asked about pairs of points determining a fixed distance, there are a number of well-studied variants involving $k$-tuples of points determining other values, such as direction, angle, or dot product. Moreover, there are many different settings besides finite point sets in $\mathbb R^2$, such fractal sets, higher dimensions, or even different types of vector spaces and modules over finite fields and rings. For some examples and discussions, see \cite{CEHIK, CHISU, HIKR, IS, Vinh} and the references contained therein. Also, we will assume that the dot products discussed below are all nonzero, unless otherwise specified, as zero dot products can exhibit quite distinct geometric structure. See \cite{GPRS, IS, KMS} for more on the unique geometric properties of zero dot products.

We highlight for the moment on the specific problem of the number of distinct dot products determined by a large finite point set in the plane. Given a set of $n$ points in the plane, a straightforward argument using the celebrated Szemer\' edi-Trotter theorem (from \cite{ST83}) guarantees the existence of at least a constant times $n^\frac{2}{3}$ distinct dot products. We prove this below as Corollary \ref{ezdp}. This was the best-known bound for many years until the recent developments by Hanson, Roche-Newton, and the fourth listed author in \cite{HRS}, where an exponent of $\frac{2}{3}+\frac{1}{2739}$ was achieved by relying on tools from additive combinatorics.
%\newpage
\subsection{Setting}
In this note, our primary focus will be trees. We now describe the setup in detail. For a given large finite point set $E\subseteq \mathbb R^d,$ with $d\geq 2,$ we consider $G_E$ the complete edge-weighted graph whose vertices are points from $E$ and whose edge weights are determined by the dot product between the associated vertices. Now, given a tree $T$, we estimate two things: how many different sequences of edge weights correspond to instances of $T$ as a subgraph in $G_E$, and how many times a specific sequence of edge weights occurs as the weights for a copy of $T$ in $G_E.$ The former is analogous to the distinct distances problem, while the latter is akin to the unit distance problem.

To this end, we introduce some notation. Given a tree $T$ with vertices $(v_1, \dots, v_{k+1})$ we denote its $k$ edges by $\{\{v_{i_1},v_{i_2}\},\{v_{i_3},v_{i_4}\},\dots,\{v_{i_{2k-1}},v_{i_{2k}}\}\}$, where the indices for each pair of vertices corresponding to an edge, $\{v_{i_{2j-1}},v_{i_{2j}}\},$ are indexed in lexicographical order. Now, given a sequence of real number weights $\vec w:=(\alpha_1,\ldots,\alpha_k)$, we can consider the weighted tree $T_{\vec w}$ where the edge connecting $v_{i_{2j-1}}$ to $v_{i_{2j}}$ has weight $\alpha_j$, for $j=1, \dots, k.$ In what follows, we will usually be concerned with edge weights equal to the dot product determined by the pair of points corresponding to the pair of vertices.

There is an abundance of literature on the case that our tree $T$ is a path of length $k,$ sometimes called a $k$-chain. See for example \cite{BIT}, \cite{FK}, and \cite{PSS} for distance chain bounds, and see \cite{BCLS} and \cite{KMS} for results related to dot product chains. More general trees have also been widely studied, but we direct special attention to \cite{OT}, by Ou and Taylor, and later, \cite{PST}, by Dung, Pham, and the fourth listed author, for motivating the results here.

Throughout this paper, we assume that $k$ is like a constant compared to the size of any subset $E$. To ease exposition, we use the asymptotic symbols $X\lesssim Y$ if $X=O(Y),$ and $X\approx Y$ when $X = \Theta(Y).$ Moreover, we write $X \gtrapprox Y$ when for every $\epsilon >0,$ there exists a constant $C_\epsilon > 0$ such that $X \gtrsim C_\epsilon q^\epsilon Y.$

\subsection{Main results}

Our main result follows the ideas in \cite{OT}. It gives lower bounds for the number of distinct sets of dot products much arise as edge weights for trees in a large finite point set in the plane. We note that there is a hypothesis about the number of points along a line through the origin, as there is a technical obstruction from certain special cases of point sets which we discuss in further detail below.

\begin{theorem}\label{main1}
Given a large finite set of points $E\subseteq \mathbb R^2,$ and a tree $T$ on $k$ edges, the number of distinct $k$-tuples of dot products, $\vec w$, for which the tree $T_{\vec w}$ is realized in $E$ is at least $\gtrsim n^\frac{2k}{3}.$
\end{theorem}

The key to proving this result is a pinned dot product result between two sets, stated precisely below as Theorem \ref{pinned}. While this result does not immediately extend to to higher dimensions, we can still obtain a nontrivial pinned estimate.

\begin{theorem}\label{highDim1}
Given a large finite set of $n$ points $E\subseteq \mathbb R^d,$ we have that there exists a point $x \in E$ for which
\[\left|\{x\cdot y:y\in E\}\right| \gtrsim n^\frac{2}{2d-1}.\]
\end{theorem}

The other result we have gives lower bounds on how often a particular set of dot product edge weights can occur in some specific point set constructions. It is based on two different constructions which have different strengths and weaknesses. 

\begin{theorem}\label{main2}
Given a tree $T$ on $k$ edges, and a large finite natural number $n,$ there exists a $k$-tuple of dot products $\vec w$ and set of $n$ points in $\mathbb R^d$ that exhibits
\[\gtrsim \max\left\{n^{1+\frac{k(d-1)}{(d+1)}},n^{\left\lceil\frac{k+1}{2}\right\rceil}\right\}\]
copies of $T_{\vec w}.$
\end{theorem}

We can see that depending on $k$ and the ambient dimension, one or the other bound may dominate. To compare this with known results, we can consider a special case. Suppose $T$ is a perfect binary tree of height $h$. It is known that this tree has $2^{h+1}-1$ vertices, meaning it has $2^{h+1}-2$ edges. So if we set $k=2^{h+1}-2,$ Theorem \ref{main2} guarantees the existence of a set of $n$ points and a $k$-tuple of dot products $\vec w$ in the plane that has $\gtrsim n^{2^h}$ copies of $T_{\vec w}$. This agrees with the power of $n$ of the upper bounds on the same, given in \cite{GPRS}, as the following result will show, showing that both bounds are tight. We note here that in \cite{GPRS}, Theorems 1.12 and 1.13 are correct for $c$-ary trees when $c=2,$ but have some superficial typos when $c\neq 2.$

\begin{theorem}\label{old}[Theorem 1.13 from \cite{GPRS}]
Given a large finite set of $n$ points $E\subseteq \mathbb R^2,$ a complete balanced binary tree $T$ of height $h$, and a $\left(2^{h+1}-2\right)$-tuple of dot products $\vec w,$ there cannot be more than $\lesssim n^{2^h}$ copies of $T_{\vec w}$ determined by points in $E$.
\end{theorem}

The structure of the rest of the paper is as follows. We first acquaint the reader with our primary tools, then prove the main technical theorem, which is a pinned estimate. We then show how to use this to prove our first main result, which is a lower bound on distinct $k$-tuples of dot products determined by trees isomorphic to a given tree. Finally, we conclude with two constructions that together establish our second main result, a lower bound on how often a particular $k$-tuple of dot products determined by trees isomorphic to a given tree can occur.

\section{Proofs}
\subsection{Tools}

Given a point $p=(p_1,p_2)\in \mathbb R^2\setminus\{(0,0)\},$ the set of points $q\in \mathbb R^2$ that determine dot product $\alpha$ with $p$ form a line whose slope is perpendicular to the line through both $p$ and the origin. To see this, we simply recall the definition of dot product and rearrange the equation $p_1q_1+p_2q_2 = \alpha.$ We sometimes call this the $\alpha$-line of $p$, and denote it $\ell_\alpha(p).$ A direct consequence of this definition is that $p\neq q$ implies $\ell_\alpha(p)\neq \ell_\alpha(q).$ We now recall the Szemer\' edi-Trotter theorem from \cite{ST83}. Given a point $p$ and a line $\ell$, we call the pair $(p, \ell)$ an incidence if the point $p$ is on the line $\ell.$ 

\begin{theorem}\label{ST}
Given a set of $n$ points and $m$ lines in the plane, the number of incidences between the points and lines is at most $$\lesssim (mn)^\frac{2}{3}+n+m.$$
\end{theorem}

We now give the following well-known corollary alluded to above to illustrate some of the basic ideas to come.

\begin{corollary}\label{ezdp}
Given a large finite set of $n$ points $E\subseteq \mathbb R^2,$ the point pairs determine at least $n^\frac{2}{3}$ distinct dot products.
\end{corollary}

\begin{proof}
Fix any $\alpha \neq 0,$ construct the set of $\alpha$-lines $\ell_\alpha(p)$ for every $p\in E.$ Any time a point $q$ is on an $\alpha$-line $\ell_\alpha(p)$, we have that $p\cdot q = \alpha,$ and vice versa. So the number of times that the dot product $\alpha$ occurs between pairs of points in the set $E$ is exactly the number of incidences between points and $\alpha$-lines. Since there were $n$ points, and each determines a unique $\alpha$-line, we have $n$ lines. So by Theorem \ref{ST}, we have that $\alpha$ can occur no more than $\lesssim \left(n^2\right)^\frac{2}{3}= n^\frac{4}{3}$ times. Since this is true for any $\alpha\neq 0,$ and there are $n^2$ distinct pairs of points, we get that the number of distinct dot products must be at least
\[ \frac{n^2}{n^\frac{4}{3}} \gtrsim n^\frac{2}{3}.\]
\end{proof}

Corollary \ref{ezdp} shows that there are many distinct dot products, but it does not show that any particular point determines many distinct dot products. For this, we introduce another concept. Given a point $p$ and a set of points $E,$ we define the {\it pinned} dot product set $\Pi_p(E),$ to be the set of dot products determined by $p$ and the points in the set $E$. In this case, we call the point $p$ a {\it pin}. Specifically,
\[\Pi_p(E):=\{p\cdot q: q\in E\}.\]

\subsection{The pinned result}

We now state a pinned version of Corollary \ref{ezdp}.

\begin{theorem}\label{pinned}
Given two subsets $E$ and $F$ of $\mathbb R^2$ consisting of $n$ points each, where $E$ has no more than $\gtrsim n^\frac{2}{3}$ points on any line through the origin, we have that there is a subset $E'\subset E$ of size $|E'|\geq \frac{1}{2}n$ with the property that for any $x\in E',$ the number of distinct dot products determined by $x$ and the points in $F$ is at least $\gtrsim n^\frac{2}{3}.$ Moreover, if $E=F,$ we get the same results without the hypothesis about the number of points on any line through the origin.
\end{theorem}

The proof we give essentially comes from Section 9.2 in \cite{GIS}, but Theorem \ref{pinned} highlights the pinned nature of the result, which is not explicitly stated in the original text. Moreover, we slightly modify the proof to get a result for two sets of points instead of one, which is what we will need later. Unfortunately, because we are specifying which set the good pin must come from, this brings in the possibility of pathological pairs of point sets. For example, if $E$ is a set of $n$ points on a line through the origin and $F$ is a set of $n$ points on a line perpendicular to the line containing $E$, then we would have that each point $x\in E$ we have $|\Pi_x(F)|=1.$ However, if $E=F$ above, we will see in the proof that we can drop the hypothesis that no line through the origin has more than $\gtrsim n^\frac{2}{3}$ points of $E.$

In order to prove this result, we will rely on a version of the so-called crossing number lemma. In order to state the lemma, we remind the reader of two concepts. First, a multigraph is a generalization of a graph allowing for multiple distinct edges connecting pairs of vertices. We refer to the number of distinct edges connecting a given pair of vertices as the edge multiplicity of that pair of vertices. Second, given a graph $G$, its crossing number, denoted $cr(G)$ is the minimum number of pairs of edges that cross taken over any drawing of $G$ in the plane. So, a planar graph has crossing number zero, but any graph that is not planar has a positive crossing number. We now state a multigraph version of the crossing number lemma, from Sz\' ekely in \cite{Szek}.
\begin{lemma}\label{CNL}
Given a multigraph $G$ with $v$ vertices, $e$ edges, and a maximum edge multiplicity of $m$, satisfying $e>5mv,$ we have that
\[cr(G) \gtrsim \frac{e^3}{mv^2}.\]
\end{lemma}

With this in tow, we prove Theorem \ref{pinned}.

\begin{proof}
%We begin with a technical reduction, and assume that $E$ is contained in the first quadrant. If not, we simply find whichever quadrant has the most points and rotate $E$ and $F$ so that at least a quarter of the points of $E$ are in the first quadrant, and proceed with the rest of the proof, losing at most a constant multiple of $\frac{1}{4}.$
We will also assume that the origin is not in $E$ or $F$, as it can only determine the dot product zero anyway. Construct a graph, $G$, whose vertex set is the point set $E\cup F$. Note, although we defined a graph $G_E$ related to a point set $E$ above, this graph will have very different edges than $G_E$. For every point $p\in E$, draw the $\alpha$-lines for each $\alpha\in\Pi_p(F).$ By definition, for each $p$, the associated lines drawn will be parallel to one another. Now, for every such $\ell_\alpha(p)$ with more than one point on it, we connect vertices by an edge if their corresponding points are consecutive along that line. Notice that we will disregard such lines with exactly one point on them. However, if $p$ and $q$ lie on the same radial line, it is possible that $\ell_\alpha(p)=\ell_\beta(q).$ In this case, the construction of the graph $H$ could lead to multiple edges between some pairs of vertices. Any time that there are multiple edges between a pair of points, $r$ and $s$, draw each edge as a smooth curve, $\mathcal C,$ connecting the appropriate points close enough to the line segment, $\mathcal L,$ connecting those points so that there are no other points from $E$ in the region bounded by $\mathcal C, \mathcal L,$ and the points $r$ and $s.$

%\begin{figure}
%\centering
%\includegraphics[scale=1]{Erdos6-2.eps}
%\caption{Consider the leftmost radial line. We draw parallel lines, each of which is perpendicular to the leftmost radial line, which cover the points. The curved arcs represent edges drawn between consecutive points on these parallel lines.}
%\label{dp2}
%\end{figure}

We now deal with a possible degenerate case. Suppose there is a point $p$ whose $\alpha$-lines contribute much fewer than $n$ edges. In this case, this means that for most of the values of $\alpha\in\Pi_p(F)$, we have that $\ell_\alpha(p)$ only has one point. That would imply that $p$ determines $\gtrsim n$ distinct dot products with points in $F$, and we would be done. So we can assume that it takes about $n$ edges for each point in $E.$ Therefore, there must be about $n^2$ edges in $H$.

Define $t$ to be the maximum number of distinct dot products determined by any point in $p\in E$ with points in $F$. That is, we define
\[t := \max_{p\in E}\{p\cdot q:q\in F\}.\]
The the bulk of rest of the proof is devoted to showing that $t\gtrsim n^\frac{2}{3}.$ For any fixed radial line, $\ell$, the vertices of consecutive points along each of the parallel lines perpendicular to $\ell$ will be connected by as many edges as there are points on $\ell$. Because of our hypothesis on the set $E$, we know $\ell$ has no more than $n^\frac{2}{3}$ points from $E$ on it. Thus, in our graph, the maximum edge multiplicity will be $m\lesssim n^\frac{2}{3}$.

We pause here to note a slight change that applies to the case where $E=F$ to handle the `moreover' part of the theorem. Namely, if we were to have $E=F,$ we would not need the hypothesis on the number of points on a radial line, as if there were more than $n^\frac{2}{3}$ points on a radial line, any of those points would determine at least as many distinct dot products as are on the radial line, with the points on the radial line in question, so each such point would automatically satisfy the conclusion of the theorem statement.

Now that we know the number of vertices, edges, and maximum edge multiplicity, we can try to apply Lemma \ref{CNL}. Recall that we need $e>5mv$ to apply the lemma. If this condition were to fail, we would have $e\leq 5mv$, but plugging in our values of $e, m,$ and $v$ would give us $n^2 \lesssim 5 n^\frac{2}{3} n$ which is a contradiction. Therefore, we can proceed assuming that this condition holds, and by plugging in values of $e, m,$ and $v,$ we are left with
\[cr(G) \gtrsim \frac{e^3}{mv^2} \gtrsim \frac{n^6}{n^\frac{2}{3}n^2} \gtrsim n^\frac{10}{3}.\]

Next, we will pair this with an upper bound for the crossing number of $G.$ Any crossing between edges only occurs when a line perpendicular to the radial line of one point, Say $p$, crosses a line perpendicular to the radial line of another point, say $q.$ Since each point has no more than $t$ such associated parallel lines (its $\alpha$-lines), each pair of points can contribute at most $t^2$ crossings. Because we know that there are about $n^2$ different pairs of points, we are guaranteed that the total number of crossings is bounded above by $\lesssim n^2t^2$. Now, while we don't know what the minimum number of crossings in any drawing of $G$ in the plane must be, we can be sure that it is no more than $n^2t^2$, because we have found an explicit drawing with no more than $n^2t^2$ crossings. Putting the upper and lower bounds for the crossing number together:
$$n^\frac{10}{3} \lesssim cr(G) \lesssim n^2t^2.$$
This tells us that $t\gtrsim n^\frac{2}{3},$ as claimed. Now, we have shown the existence of one point $x\in E$ that determines $\gtrsim n^\frac{2}{3}$ distinct dot products with points in $F$. We call this point a {\it good pin}. To finish the proof, we just need to show that there are many good pins. To see this, define the set $E_1:=E\setminus\{x\},$ and set $F_1$ to be any subset of $F$ with $n-1$ points. We can run the same argument to get that there is a good pin $x_1\in E_1,$ that determines many distinct dot products with the points in $F_1.$ Repeat this process until we have at least $m = \lceil n/2 \rceil$ such good pins, and obtain the set $E_m.$ To finish, set $E'=E_m,$ which is of size $\geq n/2,$ and consists entirely of good pins.
\end{proof}

%\steven{It looks like we can't get the two set thing to work with this pigeonholing, so maybe we just state this as a standalone pinned argument for higher dimensions, and keep the trees results in two dimensions.}

%\begin{corollary}\label{pinnedHiD}
%    Given two large finite sets of $n$ points $E\subseteq \mathbb R^d,$ where no line through the origin has more than $n^\frac{2}{2d-1}$ points of $E$ there exists a point $x\in E$ such that $|\Pi_x(E)|\gtrsim n^\frac{2}{2d-1}.$
%\end{corollary}
\subsection{Proof of Theorem \ref{highDim1}}
With a simple pigeonholing argument, we can use Theorem \ref{pinned} to get a pinned result in higher dimensions in the case that $E=F$.

\begin{proof}
Suppose $t_1$ is the maximum number of parallel $(d-1)$-hyperplanes supported by points in $E$ that are the level sets of dot products determined by points in $E.$ That is, these are the higher dimensional analogs of $\alpha$-lines. Let $x$ be such a point determining $t_{1}$ distinct dot products with points in $E.$ Pick one of these parallel $(d-1)$-hyperplanes, with as many or more points than any of the other parallel $(d-1)$-hyperplanes and call it $P_{1}$. By the pigeonhole principle, we can see that it has at least $\gtrsim n /t_1$ points of $E$. Now we repeat this process within $P_1.$ Specifically, we select $x_{2}\in P_1$ be a point determining $t_{2}$ distinct dot products, so that no other points in $P_{1}$ determine more dot products with points in $P_1.$ As before, pick a $(d-2)$-hyperplane $P_{2}$, that contains greater than $(n/t_{1})/t_{2}$ points of $E$. We continue doing this, getting successive hyperplanes of one dimension lower until $P_{d-2}$ is a plane with $(n)/(t_{1}t_{2}...t_{d-2})$ points. By appealing to Theorem \ref{pinned}, we see that this plane has a point $x_{d-2}$ that determines $(n/(t_{1}t_{2}...t_{d-2}))^{\frac{2}{3}}$ distinct dot products.

Now we collect what we know about each of the points $x_1, \dots, x_{d-1}.$ Namely, $x_1$ determines $t_1$ distinct dot products with the points of $E$, $x_2$ determines $t_2,$ distinct dot products with the points of $E$, and for all other $j<d-1,$ we have that $x_j$ determines $t_j$ distinct dot products with the points of $E$. Finally, $x_{d-1}$ determines is at least $\gtrsim ((n)/(t_{1}t_{2}...t_{d-1}))^{\frac{2}{3}}$ distinct dot products with the points in $E$. Putting all of this together, we have
\begin{equation}\label{hiDimPinned}
\max_{x\in E}\left|\Pi_x(E)\right|\gtrsim \max\left\{t_1, t_2, \dots, t_{d-2}, ((n)/(t_{1}t_{2}...t_{d-2}))^{\frac{2}{3}}\right\}.
\end{equation}
Now we consider the case when all of the components are approximately equal, namely if
\[t_1 \approx t_2 \approx \dots \approx ((n)/(t_{1}t_{2}...t_{d-2}))^{\frac{2}{3}}.\]
In this case, we get that
\[t_1\approx \left(\frac{n}{t_1^{d-2}}\right)^\frac{2}{3},\]
which gives us 
\[\max_{x\in E}\left|\Pi_x(E)\right| \approx t_1\gtrsim n^\frac{2}{2d-1}.\]
Now we deal with the case that one of the $t_j$ is significantly bigger one of the others. It is straightforward to see that this would only increase the lower bound for $\left|\Pi_x(E)\right|$. The same is true if the final term, $((n)/(t_{1}t_{2}...t_{d-2}))^{\frac{2}{3}},$ is significantly larger than the $t_j.$ To prove this rigorously, one can set up a Lagrange multiplier argument, or take logarithms of each term in base $n$ and use linear programming.

%Using our new number of points, $\frac{n}{t_{d-1}}$ for our proof in two dimensions, we can use basic pigeonholing to create a construct for higher dimensions. What this results in is the quantity of our points n, over our total number of hyper-planes, let's say $t^{k}$ taken to the $\frac{2}{3}$'s power. 
%    $$(\frac{n}{t^{k}})^{\frac{2}{3}}$$
%Let's assume that the exponent k is equal to $d-2$ where d is our number of dimensions. We can confirm this is the case by checking whether we get agreeing results in 2-dimensions where our exponent k would become 0. This, by the properties of exponents means that our denominator $t^{k}$ would become 1 which results in simply $(n)^\frac{2}{3}$ and perfectly agrees with the pinned proof of Corollary 1.

%Moving forwards, we can solve for our variable t by setting it equal to our new equation. This creates the function $t = (\frac{n}{t^{k}})^{\frac{2}{3}}$. Performing basic operations on this, we are able to determine that our t is equal to $n^{\frac{2}{2d-1}}$. This continues to agree with our previous results where if you solve using d=2, the returned value is $n^{\frac{2}{3}}$. 
%Now that we have proven this to align with what has been previously stated, we can apply this to our higher dimensions. Running this algorithm again but in higher dimensions we begin to get differing results. Using d=3 we return a value of $n^{\frac{2}{5}}$, d=4 gives us $n^{\frac{2}{7}}$ and so on and so forth. 

To summarize, we are able to pigeonhole our previous argument onto higher dimensions by running the algorithm on each hyperplane determined by our points. In general, our lower bound for higher dimensions is:
$$\max_{x\in E}\left|\Pi_x(E)\right|\gtrsim n^{\frac{2}{2d-1}}.$$
\end{proof}

\subsection{Proof of Theorem \ref{main1}}

This argument closely follows the arguments in \cite{PST}, by Dung, Pham, and the fourth listed author, which was inspired by the work of Ou and Taylor in \cite{OT}. The proof of Theorem \ref{main1} will essentially follow by applying this more technical result. For brevity, we often conflate a point in $\mathbb R^2$ with its associated vertex in a graph. Given a tree $T$ with a vertex $v$, we will often refer to the pair of them by $(T,v).$

In essence, what the theorem says is that if we fix a tree $T$ and isolate a vertex $v$, then there are many points $x$ that serve as the vertex $v$ in many trees $T'$, isomorphic to $T$, and that these isomorphic copies $T'$ have a wide variation of dot product edge weights. 
\begin{theorem}\label{tech1}
Given two subsets $E$ and $F$ of $\mathbb R^2$ consisting of $n$ points each, where $E$ has no more than $n^\frac{2}{3}$ points on any line through the origin, a tree $T$ with $k$ edges, and an arbitrary vertex $v$ of $T$, we have that there is a subset $E'\subset E$ of size $|E'|\geq 2^{-k}n-o(n)$ with the property that for any $x\in E',$ the number of distinct edge weight $k$-tuples, $\vec w,$ from trees $T'_{\vec w}$ with $(T,v)$ isomorphic to $(T',x),$ and $T_{\vec w}\setminus \{x\}\subset F$ is at least $\gtrsim n^\frac{2k}{3}.$
\end{theorem}

\begin{proof}
We now prove Theorem \ref{tech1} by induction on $k.$ The base case is when $k=1,$ which holds by applying Theorem \ref{pinned}. We assume that the conclusion holds for any tree on $k$ vertices, and then show that it also holds for any tree on $k+1$ vertices. We will do this by splitting into two cases: the case where $v$ is a leaf of $T$ (only adjacent to one other vertex of $T$) and the case where $v$ is not a leaf of $T$.

If $v$ is a leaf of $T,$ then it has a unique adjacent vertex, which we call $u.$ Notice that the tree $S:=T\setminus {v}$ has $k$ edges, so we can apply the induction hypothesis to get that there is a subset $E'\subset E$ with size at least $2^{-k}n$ such that for all $x\in E',$ there are $\gtrsim n^\frac{2k}{3}$ distinct dot product edge weight $k$-tuples $\vec w$ for trees $S'_{\vec w}$ with $x$ as a vertex so that $(S'_{\vec w},x)$ is isomorphic to $(S,u).$ Recall that for tree isomorphism here, we are not concerned with the edge weights. We are showing that there are many different trees with the same general shape, but different edge weights. Next, we can apply Theorem \ref{pinned} to the set $E'$ and any subset of $F'\subseteq F$ of size $|E'|$ to get that there a subset $E''\subseteq E'$ of size at least $|E'|/2,$ so that every point $x\in E''$ determines at least $\gtrsim n^\frac{2}{3}$ distinct dot products with points in $F'.$ Putting this all together, we have that there is a set, $E''$ of size at least $|E'|/2 \geq 2^{-(k+1)}n$ so that any point $x\in E''$ satisfies the conclusion of the theorem. This handles the case that $v$ is a leaf.

We now turn our attention to the case that $v$ is not a leaf. In this case, we will break up $T$ into two trees that share only the vertex $v.$ So $T_1$ is a tree with at least one edge and $v$ as a vertex, $T_2$ is also a tree with at least one edge with $v$ as a vertex, and the union of the two trees is $T$. Suppose the number of edges in $T_1$ is $k_1$, and the number of edges in $T_2$ is $k_2$. As $k$ is the total number of edges in $T$, and they are all accounted for without repeating by considering both $T_1$ and $T_2$, we can see that $k=k_1+k_2.$

\begin{minipage}{4.25in}
%\centering
\hskip6em\includegraphics[scale=.33]{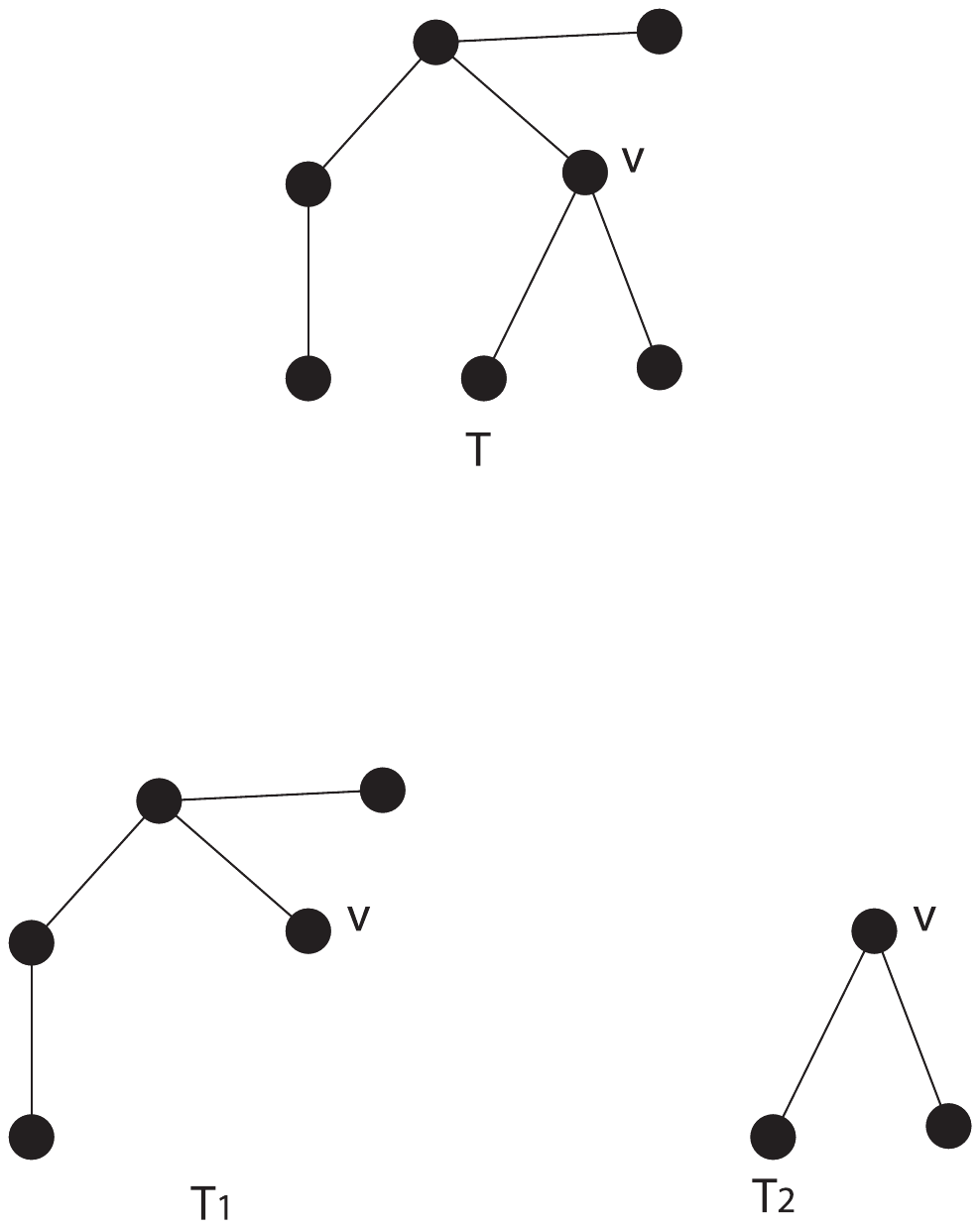}\label{dpTreesFig1}\\
{\bf Figure 2.4.1:} Here we can see the tree $T$ with vertex $v$ decomposed into $T_1$ and $T_2$, both of which also have vertex $v$.\\\\
\end{minipage}

Next, we arbitrarily partition $E$ into two subsets of equal size (either both $n/2$ or $(n+1)/2$ and $(n-1)/2$), so we have
\[E = E_1\sqcup E_2.\]
Next, we similarly partition $F$ into two subsets of nearly equal size, to get $F_1$ and $F_2$. We will now perform the following steps for both $E_1$ and $E_2$ separately. In what follows, we will often need various sets to be the same size, but they could be off by one. If this is the case, we merely delete a point from the larger set as needed providing a negligible error considering the fact that $n$ is large.

Because $k_1\leq k,$ we can apply the induction hypothesis to $E_i$ and $F_1$ with $(T_1,v)$ to find a subset $E_i'\subseteq E_i$ of size at least $2^{-k_1}|E_i|$ with the property that for every $x\in E_i'$, there are at least $\gtrsim n^\frac{2k_1}{3}$ distinct dot product $k_1$-tuples $\vec w_1$ corresponding to trees isomorphic to $T_1$ with $x$ serving as the vertex $v$ and the other vertices coming from points in $F_1$. Now, we find an arbitrary subset $F_2'\subseteq F_2$ of size $|E_1'|.$ Because $k_2\leq k,$ we can again appeal to the induction hypothesis for the sets $E_i'$ and $F_2'$ with the tree $T_2$ and vertex $v.$ This guarantees the existence of a set $E_i''\subseteq E_i'$ whose size is at least
\begin{equation}\label{Esize}
|E_i''|\geq 2^{-k_2}|E_i'| \geq 2^{-k_2-k_1}|E_i| = 2^{-k}|E_i|\geq 2^{-k-1}n - 1,
\end{equation}
with the property that for every $y\in E_i''$, there are at least $\gtrsim n^\frac{2k_2}{3}$ distinct dot product $k_2$-tuples $\vec w_2$ corresponding to trees isomorphic to $T_2$ with $y$ serving as the vertex $v$ and the other vertices coming from points in $F_2'.$ Because $E_i''\subseteq E_i',$ each of these points $y\in E_i''$ also determines at least $\gtrsim n^\frac{2k_1}{3}$ distinct dot product $k_1$-tuples $\vec w_1$ corresponding to trees isomorphic to $T_1$ with $y$ serving as the vertex $v$. Because we chose $F_1$ disjoint from $F_2,$ we know that the points in $F_1$ that serve as vertices for isomorphic copies of $T_1$ are disjoint from the points from $F_2$ that serve as vertices for the isomorphic copies of $T_2.$

We now generalize a term from before. Here, a {\it good pin} is a choice of $v$ determining sufficiently many distinct dot product tuples from their respective dot product trees, in analogy to the same term applied to single dot products in the proof of Theorem \ref{pinned}. Recall that we run the above argument for both $E_1$ and $E_2$, which were disjoint. So we can add the good pins $v$ coming from $E_1$ to the good pins $v$ coming from $E_2$ and appeal to \eqref{Esize} to get a lower bound on the number good pins to be
\[|E'|\geq |E_1''\sqcup E_2''| \geq \left(2^{-k-1}n-1\right) + \left(2^{-k-1}n-1\right) = 2^{-k}n-2 = 2^{-k}n-o(1),\]
as claimed.

To finish, notice that by definition, the dot product $k$-tuple $\vec w$ is what we get by combining $\vec w_1$ and $\vec w_2.$ Moreover, $F_1$ and $F_2$ were disjoint, so the total number of distinct dot product $k$-tuples $\vec w$ determined by isomorphic copies of $T$ with $v$ coming from $E$ and the other vertices coming from $F$ is at least the product of the number of dot product $k_1$-tuples $\vec w_1$ associated to isomorphic copies of $T_1$ from $F_1'$ times the number of dot product $k_1$-tuples $\vec w_2$ associated to isomorphic copies of $T_2$ coming from points in $F_2'.$
\begin{equation}\label{kSum}
\gtrsim n^\frac{2k_1}{3}\cdot n^\frac{2nk_2}{3} =n^\frac{2(k_1+k_2)}{3} = n^\frac{2k}{3},
\end{equation}
finishing the induction step, and therefore the proof of Theorem \ref{tech1}.
\end{proof}

To complete the proof of Theorem \ref{main1}, we need to look at lines through the origin. Now, if there is no radial line with more than $n^\frac{2}{3}$ points on it, we can repeatedly apply Theorem \ref{tech1} with $F=E$ to get the desired result. If there are radial lines with too many points, we have a decision to make. If half of the points in $E$ lie on radial lines with fewer than $n^\frac{2}{3}$ points, then we just apply Theorem \ref{tech1} with $E$ and $F$ consisting of these points. So all that remains is the case that at least half of the points are lying on radial lines with more than $n^\frac{2}{3}$ points each.

In this case, fix a radial line $\ell$ with more than $n^\frac{2}{3}$ points on it, then we can see that there are at least $n^\frac{2k}{3}$ different trees by selecting distinct choices of $x_i$ from $E_i\cap \ell.$ Without loss of generality, suppose $\ell$ is the $x$-axis. Then the points in $E_i$ have coordinates $(a_{i,j},0)$ where $j=1, \dots, n^\frac{2}{3}/k,$ and the $a_{i,j}$ are all distinct. Define
\[A_i:=\{a_{i,j}: j=1, \dots, n^\frac{2}{3}/k\}.\]
Now we can see that the different edge weights of our trees correspond to the elements of the Cartesian product of product sets of the form
\[A_iA_{i'}:=\{a_{i,j}a_{i',j}:a_{i,j}\in A_i, a_{i',j'}\in A_{i'}\}.\]
We now proceed by induction on $k.$ It is trivial to see that $|A_1A_{2}|\geq|A_1|,$ as one need only consider the products of a fixed element of one set with each of the distinct elements of the other. So the claim is true for $k=1.$ Now, suppose that we have $\gtrsim n^\frac{2(k-1)}{3}$ distinct $(k-1)$-tuples coming from the product sets $A_iA_{i'}$. We just need to show that these will lead to at least $\gtrsim n^\frac{2k}{3}$ distinct $k$-tuples when we factor in the final product set $A_{i''}A_{k+1},$ where the vertex numbered $k+1$ is a leaf adjacent to the vertex of index $i''$ that has already been dealt with. Now, for every choice of $a_{i'',j},$ there are $\gtrsim n^\frac{2}{3}/(k+1)$ distinct choices of $a_{k+1,j'}$, each leading to a different product, meaning there are at least $n^\frac{2}{3}/(k+1)$ distinct products for any of the $\gtrsim n^\frac{2(k-1)}{3}$ trees already determined, and we are done.

\subsection{Proof of Theorem \ref{main2}}

Theorem \ref{main2} will follow from the following two constructions, which we state separately as propositions. The first one is adapted from a construction given in \cite{KMS}.

\begin{proposition}\label{KMSprop}
    Given a tree $T$ on $k$ edges, there exists a set of $n$ points in $\mathbb R^d$ and a $k$-tuple of dot product edge weights $\vec w$ so that there are $\gtrsim n^{\lceil (k+1)/2\rceil}$ isomorphic copies of $T$ in the set whose edge weights are $\vec w.$
\end{proposition}

The second proposition is adapted from a construction given in \cite{IS20}.

\begin{proposition}\label{ISprop}
    Given a tree $T$ on $k$ edges, there exists a set of $n$ points in $\mathbb R^d$ so that there are $\gtrsim n^{1+k(d-1)/(d+1)}$ isomorphic copies of $T$ in the set whose edge weights are $1.$
\end{proposition}

Combining these two yields Theorem \ref{main2}. We now prove each of the propositions.

\subsubsection{Proof of Proposition \ref{KMSprop}}

This construction is a modification of a chains construction in \cite{KMS}. It can be easily modified any number of ways, but we restrict to one way that is relatively easy to communicate, while still showing the essential components that one might want to modify. We begin by 2-coloring the vertices tree $T$, which we can do because the chromatic number of any tree is at most 2. Call the two bipartition sets $U$ and $V$, and suppose without loss of generality that $|U|\geq |V|$. Now let $k_1$ denote the size of $U$ and let $k_2$ denote the size of $|V|.$ So $k_1+k_2=k+1.$ Recall that by definition, we know $k_1=|U|\geq \lceil k/2 \rceil.$ Now we proceed with the following algorithm. Pick any vertex in $U$ and call it $u_1$.

We will now put points on the standard $xy$-coordinate plane. Place $\lfloor(n-k_2)/k_1\rfloor$ points along the line $x=1$ in the plane. Now suppose that $u_1$ has $a_1$ neighbors in $T$ called $v_{1,1}$ through $v_{1,a_1}$. For each such neighbor $v_{1,j}$, place a point at $(1+j,0)$. So any of the points along the line $x=1$ have dot product $1+j$ with the point at $(1+j,0).$ Now, call the $b$ neighbors of $v_{1,1}$ (other than $u_1$) $u_2, u_3, \dots, u_{b}$. For each such neighbor $u_j$, place $\lfloor(n-k_2)k_1\rfloor$ points on the line $x=1+a_1+j$. Continue by placing single points along the $x$-axis for vertices from $V$, and columns of $\lfloor(n-k_2)/k_1\rfloor$ points for each vertex from $U$, keeping track of the relative dot products. After all of the vertices in $U$ and $V$ have been processed, place any other remaining points in such a way that they do not overlap with previously placed points, until you have a total of $n$ points, and the construction is finished.

\begin{minipage}{4.25in}
%\centering
\hskip6em\includegraphics[scale=.33]{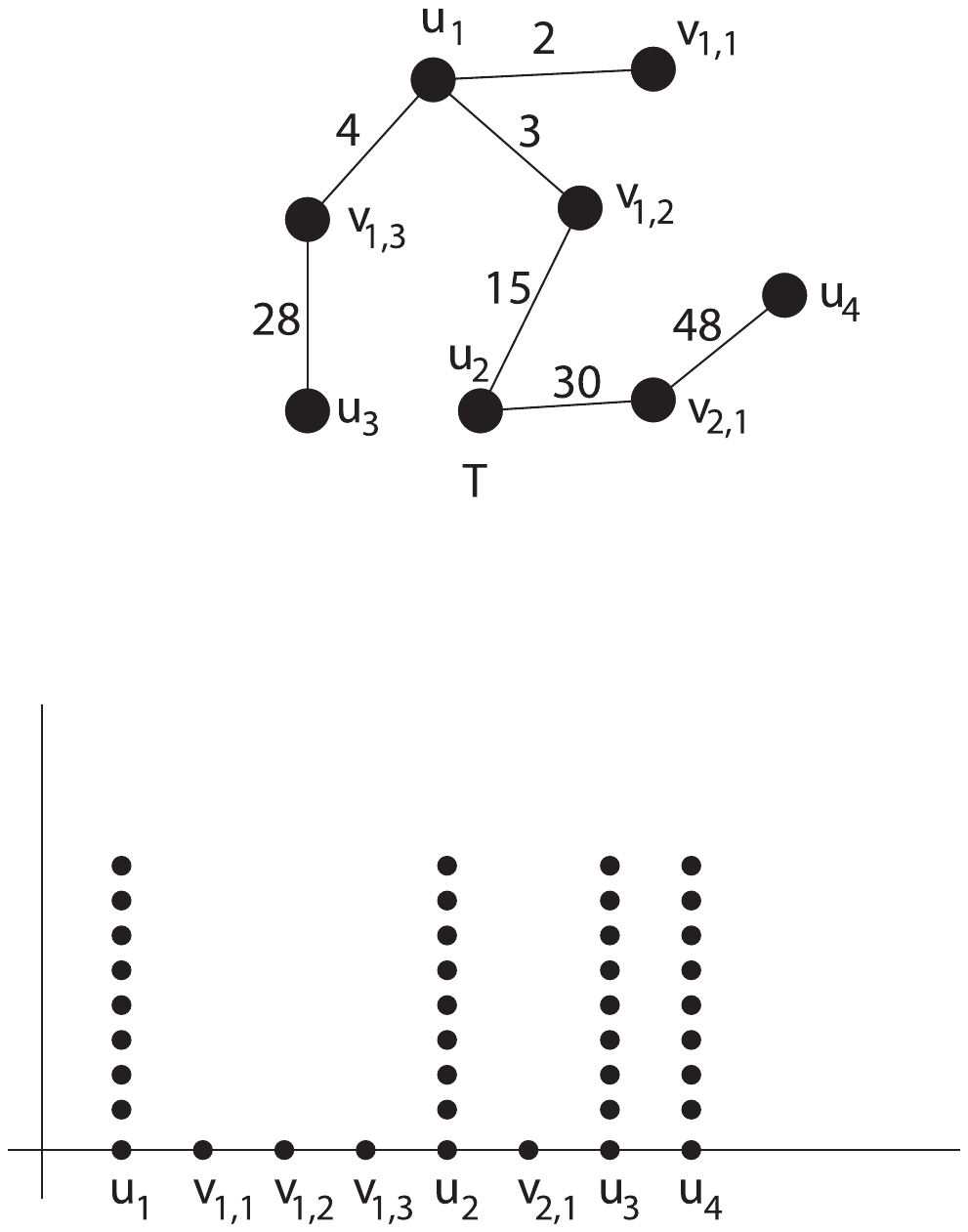}\label{dpTreesFig2}\\
{\bf Figure 2.5.1:} Here we can see the tree $T$ with vertices broken up into $U$ and $V$, along with a set of points in $\mathbb R^2$ that correspond to it. Note for example, the point at $(6,0)$ will always a choice for $v_{2,1}$, and its dot product with any point in the column of points with $x=8$, each of which is a possible choice for $u_4$, will have dot product $48$, so we have the dot product edge weight of $48$ on the edge connecting $v_{2,1}$ and $u_4$.\\\\
\end{minipage}

\begin{remark}
In the proof, we partitioned the vertices into two sets, $U$ and $V$. The proof actually shows that there are $n^|U|$ copies of the tree $T$, and if $|U|$ is much bigger than $\lceil k/2\rceil,$ the construction yields significantly more copies of the tree. Moreover, while the construction here uses lines, if one is working in higher dimensions, the points may be spread more evenly across higher dimensional hyperplanes. While this yields the same number of trees, point sets that are more evenly distributed could be desirable for some applications.
\end{remark}

\subsubsection{Proof of Proposition \ref{ISprop}}

This construction is essentially the discrete part of the construction in \cite{IS20}, though with some slight modifications as the constants chosen there may lead to some confusion. The new piece here is how to embed the trees. Also, in Proposition \ref{ISprop}, we stress that the dot products are always 1, as opposed to the previous construction, for Proposition \ref{KMSprop}, where our choices of dot product edge weights might have to be different. The basic idea is to construct two sets, $E$ and $F$, where every point in either set makes the dot product $1$ with many points in the other set. The set $E$ will be a kind of lattice, so it will have many points on many hyperplanes. The set $F$ will be the points that make dot product $1$ with many of the hyperplanes. We now sketch the modified construction.

For a large natural number $n$, fix $q\approx dn^\frac{1}{d+1}.$ Now define two sets of numbers:
$$A:=\left\{\frac{q+i}{dq}:i=1,\dots, q\right\},$$
and
$$B:=\left\{\frac{q^2+j}{d^2q^2}:j=1,\dots,q^2\right\}.$$
To build our lattice, $E$, we will take the $(d-1)$-fold Cartesian product of $A$ with itself, and then take the Cartesian product of this set with $B$. By construction, we can see that $|E|=q^{d-1}\cdot q^2\approx n.$ Next, we focus on the point set $F$, which will also be built using the sets of numbers $A$ and $B$.
$$F:=\left\{\left(\frac{-c_1}{b}, \frac{-c_2}{b}, \dots, \frac{-c_{d-1}}{b}, \frac{1}{b}\right):c_{j}\in A, b\in B \right\}.$$
Now we want to get a handle on point pairs whose dot product is $1$. Dot products are constant on hyperplanes, so given a $d$-tuple, $(c_1, c_2, \dots, c_{d-1},b),$ we can define the hyperplane
$$h(c_1, c_2, \dots, c_{d-1},b) := \left\{ (x_1, x_2, \dots, x_d)\in \mathbb R^d : x_d = \left(\sum_{j=1}^{d-1}c_jx_j\right) + b \right\}.$$
We will now focus on the following set of hyperplanes
$$H :=\left\{h(c_1, c_2, \dots, c_{d-1},b):c_j\in A, b\in B \right\}.$$
These hyperplanes are the sets of points that determine dot product $1$ with points in $F.$ To prove this, consider a $d$-tuple, $(c_1, c_2, \dots, c_{d-1},b)\in A^{d-1}\times B,$ and compute the dot product of the associated element of $F$ and any point on the hyperplane associated to the same $d$-tuple. To be precise, consider the point $f\in F$ given by
$$f= \left(\frac{-c_1}{b}, \frac{-c_2}{b}, \dots, \frac{-c_{d-1}}{b}, \frac{1}{b}\right) \text{ to be associated to } h_f:=h(c_1, \dots, c_{d-1},b),$$
a hyperplane in $H$, and compute its dot product with an arbitrary $x\in h_f.$ Their dot product will be
\begin{align*}
f\cdot x &= \left(\frac{-c_1}{b}, \frac{-c_2}{b}, \dots, \frac{-c_{d-1}}{b}, \frac{1}{b}\right) \cdot \left(x_1, x_2, \dots, x_{d-1}, \left(\sum_{j=1}^{d-1}c_jx_j \right)+b \right)\\
&=\left(\sum_{j=1}^{d-1}\frac{-c_jx_j}{b}\right) + \left(\sum_{j=1}^{d-1}\frac{c_jx_j}{b}\right) +\frac{b}{b}=1.
\end{align*}
We now consider any $f\in F$. The associated hyperplane, $h_f\in H,$ consists of points that have dot product $1$ with $f.$ By construction, for most\footnote{Here we notice that there could be some members of $f$ whose hyperplanes cut through the edge of $E$, and therefore hit significantly fewer than $\approx q^{d-1}$ points, but this will happen for $\lesssim |F|$ points. As this is a negligible loss, we ignore it and similar instances of it.} $f\in F,$ the associated hyperplane $h_f$ will have $\approx q^{d-1}$ points from $E$. Moreover, the dual relation also holds. That is, for most $e\in E,$ there are about $q^{d-1}$ points $f\in F$ such that $e\cdot f = 1.$

Now we embed our trees as in the proof of the previous result. For any tree $T$, we 2-color the vertices to obtain disjoint vertex sets $U$ and $V.$ Fix some vertex from $U$ to be $u_1.$ Let any point $f_1\in F$ be a choice for $u_1.$ We consider the set of $q^{d-1}$ choices of $e\in h_{f_1}$ that have dot product $1$ with $f_1.$ Suppose that $u_1$ has $a_1$ neighbors in $T$, and designate these by $v_{1,1}, v_{1,2}, \dots, v_{1,a_1}.$ For each of these, we select a different point in $h_{f_1}$. Notice that there are $q^{d-1}$ choices for each of them. Now consider the neighbors of $v_{1,1}$ in $T$, and for each representative point $e_{1,1}\in h_{f_1}$ of each of them, choose from the $q^{d-1}$ points of $F$ that make dot product $1$ with $e_{1,1}.$ Continuing in this way, each new vertex that is processed in the tree has $\approx q^{d-1}$ possible choices. Since there were $\approx n$ choices for $u_1,$ and each of the $k$ subsequent points had $q^{d-1}\approx n^\frac{d-1}{d+1}$ choices, we end up with a total of $\gtrsim n^{1+k(d-1)/(d+1)}$ isomorphic copies of $T$ in the set whose edge weights are $1$, as claimed.

\begin{remark}
We pause to note that the higher dimensional results may appear unnecessary in light of a slight modification of the construction in Remark 2 of \cite{KMS}, which can give $n^k$ trees for certain $k$-tuples of edge weights in dimensions $d\geq 3.$ However, the constructions given here are not concentrated on lines, and therefore are $s$-adaptable for a wider range of $s.$ See \cite{IS20, KMS} for more on $s$-adaptability, and how it connects to constructing fractal sets. For completeness, we give the construction for $k$-paths, and describe how to modify it for general trees.

Given $n$ and $k$, arrange $\left\lfloor n/(k+1)\right\rfloor$ points on each of the following lines perpendicular to the $x$-axis: $\{(1,y,0)\}, \{2,0,z\}, \{(3,y,0)\}, \{(4,0,z)\}\dots$ alternating whether $y$ or $z$ is the free variable on every other line. We then 2-color the path, and map the vertices of the (potentially) larger color class to the points on the lines with odd $x$ coordinates and map the vertices of the other color to the points on the lines with even $x$ coordinate. To modify this for general trees, we begin by 2-coloring the tree, then make sure we have a $y$-line for every vertex of one color and a $z$-line for every vertex of the other color. So the tree given in Figure \ref{dpTreesFig2} could be given by $\left\lfloor n/8 \right\rfloor$ points on each of the following lines:
\[\{(1,y,0)\}, \{2,0,z\}, \{(3,0,z)\}, \{(4,0,z)\},\{(5,y,0)\}, \{6,0,z\}, \{(7,y,0)\}, \{(8,y,0)\}.\]
\end{remark}

\section{Acknowlegdments}
We would like to thank Alex Iosevich, Ethan Lynch, and Caleb Marshall for valuable discussions that improved the results here and the exposition thereof.


\begin{thebibliography}{99}


%\bibitem{AB} J. Alvarez-Bermejo, J. A. Lopez-Ramos, J. Rosenthal, D. Schipani, {\it Managing key multicasting through orthogonal systems}, arXiv:1107.0586v2, (2015).

%\bibitem{Bahls} P. Bahls, {\it Channel assignment on Cayley graphs.} J. Graph Theory, 67: 169--177, (2011). doi: 10.1002/jgt.20523

\bibitem{BS} D. Barker and S. Senger, {\it Upper bounds on pairs of dot products}, Journal of Combinatorial Mathematics and Combinatorial Computing, Volume 103, November, 2017, pp. 211--224.

\bibitem{BIT} M. Bennett, A. Iosevich, and K. Taylor, {\it Finite chains inside thin subsets of ${\Bbb R}^d$,} Analysis and PDE, volume 9, no. 3, (2016). 

\bibitem{BCLS} V. Blevins, D. Crosby, E. Lynch, and S. Senger, {\it On the Number of Dot Product Chains in Finite Fields and Rings}, Nathanson M. (eds) Combinatorial and Additive Number Theory V. CANT 2021. pp. 1--20. Springer Proceedings in Mathematics \& Statistics. Springer, Cham.

\bibitem{BKT} J. Bourgain, N. Katz, and T. Tao, {\it A sum-product estimate in finite fields, and applications,} Geom. Funct. Anal. {\bf 14} (2004), pp. 27--57.

\bibitem{BMP} P. Brass, W. Moser, and J. Pach, {\it Research Problems in Discrete Geometry,} Springer (2000), 499 pp.

\bibitem{CEHIK} J. Chapman, B. Erdo\u{g}an,  D. Hart, A. Iosevich, D. Koh, {\it Pinned distance sets, $k$-simplices, Wolff's exponent in finite fields and sum-product estimates}, Math. Z. 271 (2012), no. 1-2, 63--93.

\bibitem{CS} D. Covert and S. Senger, {\it Pairs of dot products in finite fields and rings}, Nathanson M. (eds) Combinatorial and Additive Number Theory II. CANT 2015, CANT 2016. Springer Proceedings in Mathematics \& Statistics, vol 220. Springer, Cham. (Appeared 14 Jan. 2018)

\bibitem{CIP} D. Covert, A. Iosevich, J. Pakianathan, \emph{Geometric configurations in the ring of integers modulo $p^{\ell}$}, Indiana University Mathematics Journal, 61 (2012), 1949--1969.

\bibitem{CHISU} D. Covert, D. Hart, A. Iosevich, S. Senger, and I. Uriarte-Tuero, {\it An analog of the Furstenberg-Katznelson-Weiss theorem on triangles in sets of positive density in finite field geometries,}  Discrete Math. 311, no. 6, 423--430, (2011). 


\bibitem{Erd46} P.~Erd\H{o}s, {\it On sets of distances of $n$ points}, Amer. Math. Monthly {\bf 53} (1946) 248--250.

%\bibitem{ES83} P. Erd\H os and E. Szemer\'edi, {\it On sums and products of integers}. Studies in pure mathematics, 213--218, Birkh\"auser, Basel, 1983.

%\bibitem{Fal86} K. J. Falconer, {\it On the Hausdorff dimensions of distance sets}, Mathematika \textbf{32} (1986) 206--212.

%\bibitem{Fickus} J. J. Benedetto and M. Fickus, {\it Finite normalized tight frames,} Adv. Comput. Math. 18, pp. 357--385 (2003).

%\bibitem{Ford} K. Ford, {\it Integers with a divisor in an interval}, Annals of Math. \textbf{168} (2), 367-433, 2008.

\bibitem{FK} N. Frankl and A. Kupavskii, {\it Almost sharp bounds on the number of discrete chains in the plane}, Combinatorica 42 (Suppl 1), 1119--1143 (2022).

\bibitem{GIS} J. Garibaldi, A. Iosevich, and S. Senger, {\it Erd\H os distance problem,} AMS Student Library Series, 56, (2011).

\bibitem{GK} L. Guth and N. H. Katz, {\it On the Erd\H os distinct distance problem in the plane,} Annals of Math., Pages 155--190, Volume 181 (2015), Issue 1.

\bibitem{GPRS}  S. Gunter, E. Palsson, B. Rhodes, and S. Senger, {\it Bounds on point configurations determined by distances and dot products,} Nathanson M. (eds) Combinatorial and Additive Number Theory IV. CANT 2019, CANT 2020. Springer Proceedings in Mathematics \& Statistics, vol 347. Springer, Cham.

\bibitem{HRS} B. Hanson, O. Roche-Newton, and S. Senger, {\it Convexity, superquadratic growth, and dot products}, Journal of the London Mathematical Society, Volume 107, Issue 5, May 2023, pp. 1900--1923.

\bibitem{HIKR} D. Hart, A. Iosevich D. Koh, M. Rudnev, {\it Averages over hyperplanes, sum-product theory in vector spaces over finite fields and the Erd\H os-Falconer distance conjecture}, Trans. Amer. Math. Soc. 363 (2011), no. 6, 3255--3275.

%\bibitem{IJL} A. Iosevich, H. Jorati, and I. \L aba, {\it Geometric incidence theorems via Fourier analysis} Trans. Amer. Math. Soc. \textbf{361} (2009) 6595--6611.

\bibitem{IR07} A.~Iosevich and M.~Rudnev, {\it Erd\H os-Falconer distance problem in vector spaces over finite fields}, Trans. Amer. Math. Soc., 359 (2007), no.~12, 6127--6142.

%\bibitem{IRU} A. Iosevich, M. Rudnev, and I. Uriarte-Tuero, {\it Theory of dimension for large discrete sets and applications}, Math. Model. Nat. Phenom. 9 (2014), no. 5, 148--169.

%\bibitem{IRR} A. Iosevich, O. Roche-Newton, M. Rudnev, {\it On an application of the Guth-Katz theorem}, Mathematics Research Notices 18, no. 4, 1--7 (2014).

\bibitem{IS} A. Iosevich and S. Senger, {\it Orthogonal systems in vector spaces over finite fields,} Electronic J. of Combinatorics, Volume 15, December (2008).

\bibitem{IS20} A. Iosevich and S. Senger, {\it Falconer-type estimates for dot products}, Bulletin of the Hellenic Mathematical Society {\bf 64} 98--110 (2020).

%\bibitem{IS2} A. Iosevich and S. Senger, {\it Sharpness of Falconer's estimate in continuous and arithmetic settings, geometric incidence theorems and distribution of lattice points in convex domains}, arXiv:1006.1397, (2010).

%\bibitem{KS} N. H. Katz, C. Y. Shen, {\it A slight improvement to Garaev's sum product estimate.}  Proc. Amer. Math. Soc. \textbf{136} (2008), no. 137, 2499-2504.


\bibitem{KMS} S. Kilmer, C. Marshall, and S. Senger, {\it Dot product chains,} (accepted) arXiv:2006.11467 (2020).

%\bibitem{KonShk} S. Konyagin, I. Shkredov, {\it On sum sets of sets, having small product set}, arXiv:1503.05771 (2015).

%\bibitem{Solymosi} J. Solymosi, \emph{Bounding multiplicative energy by the sumset}, Adv. Math. 222 (2009), no. 2, 402--408.

\bibitem{OT} Y. Ou and K. Taylor, {\it Finite point configurations and the regular value theorem in a fractal
setting}, Indiana Univ. Math. J. 71 (2022), no. 4, 1707--1761.

\bibitem{PSS} E. Palsson, S. Senger, and A. Sheffer, {\it On the number of discrete chains}, (to appear in the Proceedings of the American Mathematical Society).

\bibitem{PST} T. Pham, S. Senger, and D. Tran, {\it Distribution of pinned distance trees in the plane $\mathbb F_p^2$}, Discrete Mathematics, Volume 346, Issue 12, December 2023, 113613.

\bibitem{SST} J. Spencer, E. Szemer\'edi, and W. T. Trotter, {\it Unit distances in the Euclidean plane} B. Bollob\'as, editor, ``Graph Theory and Combinatorics", pages 293--303, Academic Press, New York, NY, (1984).

\bibitem{Steele} J. M. Steele, {\it The Cauchy-Schwarz Master Class ICM Edition: An Introduction to the Art of Mathematical Inequalities}, Cambridge University Press, (2010).

%\bibitem{Steinerberger} S. Steinerberger {\it A note on the number of different inner products generated by a finite set of vectors,} Discrete Mathematics, 310, (2010), pp. 1112--1117.

\bibitem{Szek} L.A.~Sz\'{e}kely, {\it Crossing numbers and hard Erd\H{o}s problems in discrete geometry}, Combin. Probab. Comput. {\bf 6} (1997), no.~3, 353--358.

\bibitem{ST83}
E.\ Szemer\'edi and W.\ T.\ Trotter, {\it Extremal problems in discrete geometry}, \emph{Combinatorica} {\bf 3} (1983), 381--392.

%\bibitem{TaoRing} T. Tao, {\it The sum-product phenomenon in arbitrary rings}. Contrib. Discrete Math. 4 (2009), no. 2, 59--82.

\bibitem{Vinh} T. V. Pham and L. A. Vinh {\it Orthogonal Systems in Vector Spaces over Finite Rings}, Electronic J. of Combinatorics, Volume 19, Issue 2 (2012).

%\bibitem{VST} L. A. Vinh {\it Szemer\' edi-Trotter type theorem and sum-product estimate in finite fields} (2013).

\end{thebibliography}
\end{document}